\documentclass[12pt]{article}
\usepackage{amsmath,amssymb,amsthm,amsfonts,amscd}
\usepackage{hyperref}
\usepackage[numeric]{amsrefs}
\usepackage{rotating} 
\input colordvi
\usepackage{graphicx}

\newtheorem{thm}[equation]{Theorem}

\newtheorem{lem}[equation]{Lemma}

\newcommand{\thmref}[1]{Theorem~\ref{#1}}

\newcommand{\secref}[1]{Section~\ref{#1}}

\numberwithin{equation}{section}

\renewcommand\a{\alpha}
\renewcommand\b{\beta}

\renewcommand\d{\delta}
\newcommand\e{\varepsilon}

\newcommand\D{\Delta}

\renewcommand\D{\Delta}

\newcommand\f{\frac}
\newcommand\smallf[2]{{\textstyle{\frac{#1}{#2}}}}

\newcommand{\Z}{{\mathbb{Z}}}
\newcommand{\R}{{\mathbb{R}}}

\newcommand{\C}{{\mathbb{C}}}

\newcommand{\Q}{{\mathbb{Q}}}

\renewcommand\Re{\text{Re~}}

\renewcommand\i{^{-1}}
\renewcommand\({\left(}
\renewcommand\){\right)}
\newcommand{\ttwo}[4]{
\(\begin{smallmatrix}{#1} & {#2}
\\ {#3} & {#4} \end{smallmatrix}\)}

\newcommand{\bx}{\hfill$\square$\vspace{.6cm}}
\newcommand{\sgn}{\operatorname{sgn}}

\newcommand\srel[2]{\begin{smallmatrix} {#1} \\ {#2} \end{smallmatrix}}

\newcommand{\gobble}[1]{}
  \newcommand{\rangeref}[2]{%
    \ref{#1}--\afterassignment\gobble\fam 0\ref{#2}%
  }

\makeatletter
\def\imod#1{\allowbreak\mkern5mu({\operator@font mod}\,#1)}
\makeatother

\begin{document}

\title{Non-unitarity outside the fundamental parallelepiped}

\author{Stephen D. Miller\thanks{Supported by NSF grant DMS-2101841.}\\
Yeshiva University
}

\maketitle

\begin{center}
{\it To Wilfried Schmid, with appreciation and admiration.}
\end{center}

\begin{abstract}
One of the challenges of the unitary dual problem is the daunting number of possible representations to consider.  This article describes an approach to narrowing the search space for minimal principal series representations, in terms of the ``fundamental parallelepiped'' (or ``FPP''):~the set of linear combinations of fundamental weights with coefficients in the interval $[0,1]$.
An earlier conjecture of the author, which was rooted in work of  Barbasch and then subsequently vastly generalized by  Vogan (and recently proven by  Davis and  Mason-Brown), asserts that the FPP houses all dominant infinitesimal characters for which
minimal principal series have a unitarizable quotient.

We  present a technique to prove this ``FPP inequality'' in specific examples, demonstrated here for the split real form $E_{8(8)}$, by introducing a limiting theory of intertwining operators as $\nu$ approaches $\infty$ in directions of fundamental weights.
  This limiting theory is inspired by Wilfried Schmid's work on variation of Hodge structure.
As an application, we show that the unitary set for a particular minimal principal series consists of the closure of a single open alcove.
\end{abstract}

\section{Introduction}\label{sec:intro}

The unitary dual problem is the last major unsolved problem in the representation theory of Lie groups.  One striking feature of this long-standing challenge, which asks to classify all the unitarizable representations of a real reductive Lie group, is the sheer size of the region of representations whose unitarity status is unknown.  The focus of this paper is an inequality which a large class of unitary representations of a split, semisimple real Lie group  $G$ must satisfy.  This class, which includes  nearly all examples studied in the spectral theory of automorphic forms (such as Maass forms), can be defined as those representations not cohomologically induced from  smaller groups, or equivalently as those containing a fine $K$-type in the sense of \cite{BGG}.  Each is a quotient
 of some dominant minimal principal series.  

We begin by explaining a special case of the inequality:~an inequality discovered by Dan Barbasch in studying the spherical unitary dual.  Let $\nu$ be the infinitesimal character of a unitary spherical representation of $G$, which without loss of generality can and will be assumed to be dominant and real.  (The classification of unitary representations for arbitrary $\nu$ reduces to that for real $\nu$.) The inequality is that
\begin{equation}\label{barbasch1}
  0 \ \le \ \langle \nu,\alpha^\vee \rangle \ \le \ 1
\end{equation}
for all simple coroots $\alpha^\vee$.  The value of 1 on the right-hand side is sharp, as it is achieved  by the trivial representation.
Equivalently, $\nu$ may be written as
\begin{equation}\label{barbasch2}
  \nu \ = \ \sum c_\varpi \varpi\,, \ \  \ \text{with} \ 0 \le c_\varpi \le 1 \ \ (\forall \varpi)\,,
\end{equation}
where
$\varpi$ ranges over all fundamental weights.  We call this region the ``fundamental parallelepiped'' (abbreviated as ``FPP''), and   (\ref{barbasch1})-(\ref{barbasch2}) the ``FPP inequality''.

Using techniques geared towards this specific situation, Barbasch   proved the FPP inequality for spherical unitary representations as part of his impressive and sharper program of calculating signatures on $K$-types \cite{BarbaschSph}, utilizing his important notion of {\em petite} $K$-type \cite{petite} that appears prominently in \secref{sec:FPPproof}.  Over the past five years there have been a number of approaches to analyzing whether or not the FPP inequality extends beyond spherical representations.  To summarize, the author conjectured that the FPP inequality holds for unitary quotients of dominant minimal principal series, and proved this conjecture for exceptional groups in~\thmref{mainthm} below. David Vogan then generalized this conjecture to a statement that all failures of the FPP inequality for unitary representations are accounted for by cohomological induction, from unitary representations that do satisfy the FPP inequality.  Vogan's conjecture was later proven by Davis and Mason-Brown \cite{DM-B}.

\begin{thm}(FPP Inequality)\label{mainthm}
 Suppose $\pi$ is a unitary quotient of a minimal principal series representation of a split real form of an exceptional group.  Assume that $\pi$'s infinitesimal character $\nu$ is real and dominant.   Then $\nu$ satisfies the equivalent conditions (\ref{barbasch1})-(\ref{barbasch2}).
\end{thm}

The importance of this theorem is that it considerably reduces the search space for unitary representations, as compared to existing bounds such as the Dirac inequality.  For $E_{8(8)}$ the reduction in volume is by a factor of about 861 alone, not to mention the even greater benefit of eliminating some of the most difficult ranges of infinitesimal characters.   As evidence for its effectiveness, the FPP bound has been used  in joint work with Adams, van Leeuwen, and Vogan to now classify the unitarity of Langlands quotients from minimal principal series, for all split exceptional groups.

Because of space considerations and the fact that the recent results of   \cite{DM-B} now subsume ours, we will explain the proof of Theorem~\ref{mainthm}  only in its most difficult case of $G=E_{8(8)}$, the split real form of $E_8$.  The other exceptional groups are handled fairly similarly and the technique is surely applicable in much broader generality, though it requires a case-by-case analysis, whereas \cite{DM-B} does not.  It should be mentioned, however,  that groups such as $E_{7(7)}$, whose principal series contain multiple lowest $K$-types, and groups such as $E_{6(6)}$, whose Hermitian condition requires a nontrivial fixed-point condition,  contain some  subtleties not present for $E_{8(8)}$.

The methods of this paper and of \cite{DM-B} are both influenced, in totally different ways, by Wilfried Schmid's work on Hodge theory.  Indeed, the main tool here is the introduction of a theory of limiting structures on Hermitian forms as $\nu\rightarrow\infty$ in the direction of a fundamental weight.  This was stylistically inspired by Wilfried Schmid's seminal work on variation of Hodge structure \cite{SchmidHodge1,SchmidHodge2}, though the  aspects that survive in the presentation here are admittedly superficial.  In more detail, one of Schmid's most impressive results  shows there is a well-defined limit of Hodge structures on a family of degenerating annuli \cite{SchmidHodge1}.  Analogously, we show that Hermitian forms on minimal principal series representations have certain well-defined normalized limits, which can be computed in terms of Hermitian forms for  Levi components.  In terms of the notation and choice of Weyl group representatives given in \secref{sec:background}, our most general result can be stated as follows:

\begin{thm}(Limiting Theorem)\label{thm:limit}
For $\b\in\Sigma$ let
$L_\beta$ denote the standard Levi subgroup associated to $\beta$, and let $w_{\text{long},\b}$ denote the long element of the Weyl group of $L_\beta$.
 Fix $\nu$ such that $\langle \nu,\beta^\vee\rangle=0$.  Let $f_\bullet$ be a smooth flat section for the principal series (\ref{princseries}).
Then the limiting value of the long intertwining operator on $f_\bullet$  in the limiting direction of $\varpi_\b$ is given by
\begin{equation}\label{limittheorem}
  \lim_{s\rightarrow\infty} R(w_{\text{long}},\nu+s\varpi_\beta,\delta)f_{\nu+s\varpi_\beta}
 \ \ = \ \
\e\,\Lambda(w_\text{long}(w_{\text{long},\beta})\i)(R(w_{\text{long},\b},\nu,\d)f_\nu),
\end{equation}
where $\Lambda$ denotes the left translation operator on functions in $C^\infty(K)$ and $\e$ is an explicitly computable constant that depends only on $\beta$ and $\delta$.
\end{thm}

\noindent Thus the limiting value of the Hermitian form is computable in terms of smaller intertwining operators.  The constant $\e$ is 1 when $\delta$ is trivial, that is, for spherical principal series.  \\

  All the objects in \thmref{thm:limit} can be explicitly computed for small $K$-types.
This precise control of Hermitian form signatures has other applications, such as showing non-unitary in large subregions inside the FPP itself.  As an example, we have the following classification of the unitary set for the principal series representation $V_{\nu,\d}$ from (\ref{princseries}) induced from a particular order-two character $\d$ (see~\secref{sec:background} for notation).  Recall that the dominant Weyl chamber is partitioned into open simplices (``alcoves'') after removing the hyperplanes $\{\nu : \langle \nu,\a^\vee\rangle =n\}$, for all $\a\in\Delta_+$ and $n\in \mathbb{Z}_{\ge 0}$, that contain all possible points of reducibility for $V_{\nu,\d}$.  The  unitarity of a quotient of $V_{\nu,\d}$ is therefore constant on facets of the resulting simplicial decomposition of the dominant Weyl chamber.  The alcove touching the origin is the ``fundamental alcove'', and the full FPP itself is tiled by 696,729,600 affine translates of its closure.  In general, the unitary set is a complicated, mixed-dimensional simplicial complex formed by a finite set of facets.  The following result shows there is a special case in which it is as simple as possible.

\begin{thm}\label{thm:easyalcove}
{\bf A simple part of the unitary dual:}
Consider the  order-two character $\d_{\varpi_8}$ of the Borel subgroup of $E_{8(8)}$, which is uniquely characterized by being trivial on its intersection with the semisimple part of the standard Levi subgroup of type $E_7$. Then the quotient of a principal series $V_{\nu,\d_{\varpi_8}}$ with dominant infinitesimal character $\nu$ is unitarizable if and only if
\begin{equation}\label{fundalcove}
  \langle \nu,\alpha_{\text{high}}^\vee\rangle \ \ \le \ \ 1\,,
\end{equation}
that is, $\nu$ lies in the closure of the fundamental alcove.
\end{thm}

After background and notation is given in \secref{sec:background}, the necessary properties of intertwining operators are reviewed in \secref{sec:intertwiners}.  \secref{sec:sub:calculation} in particular discusses an explicit method to compute intertwining operators on fixed $K$-types.
The proof of the Limiting Theorem~\ref{thm:limit} is then given in \secref{sec:limittheorem}.  The FPP inequality of Theorem~\ref{mainthm} for $E_{8(8)}$ is proved in \secref{sec:FPPproof}, as is \thmref{thm:easyalcove}.  Finally, \secref{sec:spherical} uses the Limiting Theorem to give a new proof of Barbasch's theorem that the FPP inequalities hold in the spherical case.\\

{\bf Acknowledgements:} I would like to thank Jeff Adams, Dan Barbasch, Jack Buttcane,  Ezra Getzler, Michael Harris, Joseph Hundley, Peter Sarnak, Wilfried Schmid,  Freydoon Shahidi, and Akshay Venkatesh for very helpful conversations. I am indebted to David Vogan for background and advice on many important issues, and for a careful reading of the manuscript.  Zhuohui Zhang kindly supplied me with  code \cite{Zhangrepo} to compute explicit realizations of low-dimensional Lie algebra representations, without which the results of this paper would not have been possible.
I am also grateful to Lucas Mason-Brown for sharing the results from \cite{DM-B}, and for explaining what connection they may have to the arguments here.  I wish to also thank Rutgers University for access to computational resources used to double-check the calculations of \secref{sec:FPPproof}.

I particularly would like to extend my thanks to the editors for putting together this volume in honor of Wilfried Schmid, whose unique combination of patience and energy resulted in the introduction of many deep, long-gestating ideas across mathematics -- and across over half a century.  His intuition of where to dig, along with the strength to do lengthy and strenuous digging, changed the topography of the land he walked on.

\section{Lie-theoretic background and notation}\label{sec:background}

The results of our paper are for the split real group $G=E_{8(8)}$, though the general strategy applies with straightforward modifications to all split exceptional groups (and, we anticipate, to any fixed, split reductive group).  For that reason we work with the general language of a split semisimple real group, realized as the $\R$-points of a connected, reductive linear algebraic group which is defined and split over $\Q$, with an irreducible root system $\Delta$.  Let
\begin{equation}\label{Sigmadef}
  \Sigma \ \ = \ \ \{\alpha_1,\ldots,\alpha_r\}\,, \  \ \ \ r \ \ = \ \ \operatorname{rank}(G)\,,
\end{equation}
 denote a fixed choice of simple roots, and $\Delta_+$ the positive roots for this choice of $\Sigma$.  Then $\Delta=\Delta_+\sqcup \Delta_{-}$, with $\Delta_{-}=-\Delta_+$ denoting the negative roots.  There is a unique highest root $\alpha_{\text{high}}\in \D_+$ with largest sum of coefficients in its expansion in $\Sigma$.
Dual to the simple coroots $\{\alpha^\vee|\alpha\in\Sigma\}$ are the fundamental weights $\{\varpi_\alpha|\a\in\Sigma\}$.  Let $\Lambda_{\text{wt}}$ denote the weight lattice, which is the $\Z$-span of the fundamental weights.  Let $\rho$ denote the half-sum of positive roots, which is equal to the sum of the fundamental weights.

We carry over root vectors $X_{\a}$, $X_{-\a}$, and $H_\a=[X_\a,X_{-\a}]$ for each embedded root-${\frak sl}_2$ from the Chevalley structure of ${\frak g}=Lie(G)$.  Denote by $\Phi_\a:SL(2,\R)\rightarrow G$ the homomorphism whose differential sends the standard Lie algebra basis elements $\ttwo 0100$, $\ttwo 0010$, and $\ttwo{1}00{-1}$ of ${\frak sl}_2(\R)$ to $X_\a$, $X_{-\a}$, and $H_\a\in {\frak g}$, respectively.

The Iwasawa decomposition
allows us to factor $G=NAK$, where:~$N$ is the unique maximal unipotent subgroup of $G$ containing the one-parameter subgroups generated by each $X_\a$, $\a\in\Sigma$; $A\cong(\R_{>0})^r$ is the   abelian subgroup equal to the product of the one-parameter subgroups generated by each $H_\a$, $\a\in\Sigma$; and $K$ is the maximal compact subgroup containing  the one-parameter subgroups generated by each Lie algebra element $Z_\alpha:=X_{\alpha}-X_{-\alpha}$, $\alpha \in \Sigma$.

In particular, $A$ is the exponential of the Cartan subalgebra $\mathfrak{h}$ spanned by the $H_\a$; it lies in the algebraic maximal torus $T=MA$, where $M\cong (\Z/2\Z)^{r}$ is the centralizer of $A$ in $K$.  $T$ is a maximal abelian subgroup of $G$.
Let $\frak h^*_{\C}$ denote the complex span of $\Delta$ and $\langle \cdot,\cdot \rangle$  the pairing between $\mathfrak{h}_\C^*$ and $\mathfrak{h}$.  Elements $\nu$ of $\frak h^*_{\C}$ can be given coordinates through the pairings $\langle \nu,\alpha^\vee\rangle$, $\a\in\Sigma$, which are their coefficients when expressed as a linear combination of fundamental weights.
    Characters $\chi$ of $T$ have the form $\chi(me^H)=\delta(m)e^{\langle \nu,H\rangle}$, where  $m\in M$, $H\in \mathfrak{h}$, $\delta$ is a character of $M$, and $\nu\in\mathfrak{h}_\C^*$.  For brevity we will at times use the notation $a^\nu$ instead of $e^{\langle \nu,H\rangle}$ for elements of the form $a=\exp(H)$, $H\in \frak h$.

A character $\chi=\delta \times (\cdot)^{\nu+\rho}$ as above extends to $B=TN$ by defining it to be trivial on $N$.  Such characters induce to the minimal principal series representations
\begin{equation}\label{princseries}
  V_{\nu,\delta} \ \ = \ \ \operatorname{Ind}_B^G\,\(\d \times (\cdot)^{\nu+\rho}\).
\end{equation}
  For example, if $\delta$ is trivial this induction results in the spherical principal series.

Every $w\in W$ has a factorization as a reduced product $w_{\beta_1}w_{\beta_2}\cdots w_{\beta_\ell}$ of simple Weyl reflections $w_\beta$ about some $\beta\in\Delta_+$.  The unique value of $\ell$  occurring in such a reduced factorization of $w$ is the length $\ell(w)$, and is strictly maximized by the long Weyl group element $w_{\text{long}}$.  The long element has order two and is the composition of negation with a diagram automorphism.  In particular, when the Dynkin diagram has no nontrivial automorphisms (as is the case for all exceptional root systems of type other than $E_6$), $w_{\text{long}}$ acts as scalar multiplication by $-1$ on ${\frak h}^*_{\C}$.

In our setting, we may take representatives for $W$ in $G$ as follows.  First, the representative for the simple Weyl reflection corresponding to $\a\in\Sigma$ will be taken to be $w_\a:=\Phi_\a(\ttwo 0{1}{-1}0)=\exp(\pi Z_\a/2)$.  To define representatives of arbitrary elements $w\in W$, we factor $w$ as a product of minimal length in the simple Weyl reflections, and take the corresponding products of these representatives of simple reflections.  According to \cite[Lemma 83b]{Steinberg}, this procedure produces a well-defined representative of $w$ in the normalizer of $T$.  Henceforth we will identify $w\in W$ with this representative in $G$, noting that the representative of a product need not be the product of its representatives.

The conjugate $w_{\text{long}}Nw_{\text{long}}^{-1}$ is equal to the ``opposite'' unipotent radical $N_{-}$, which  contains the one-parameter subgroups generated by root vectors for $\Delta_{-}$.  Each $\beta\in\Sigma$ defines a standard maximal parabolic subgroup $P_\beta=L_\beta N_\beta$, where $L_\beta$ contains the one-parameter subgroups generated by root vectors in  the integral span of $\Sigma -\{\beta\}$, and $N_\beta\subset N$. Let $w_{\text{long},\beta} \in W$ denote the long Weyl element of $L_\beta$.  Using Kostant coset representatives, $w_{\text{long}}$ has a factorization  of the form
\begin{equation}\label{Kostant}
  w_\text{long} \ \ = \ \ w'w_{\text{long},\beta}, \qquad \ell(w_\text{long})=\ell(w')+\ell(w_{\text{long},\beta})
\end{equation}
for some $w'\in W$
(see \cite[Prop.~1.10(c)]{Humphreys}).  In particular, concatenating reduced words for $w'$ and $w_{\text{long},\beta}$ results in a reduced word for  $w_\text{long}$, and this identity in fact holds for their representatives in $G$ chosen as above.

The crux of our argument involves varying the continuous parameter $\nu$.  As such it is  useful to consider the restrictions of elements of $V_{\nu,\delta}$ to $K$, which completely determine them via the Iwasawa decomposition.  In this way one can consider ``flat sections'' $f_\nu$ (in the number-theory parlance), which have a common restriction to $K$ as $\nu$ varies.  Sometimes we denote a flat section by $f_\bullet$ when the infinitesimal character $\nu$ is not necessary to specify.  Thus, flat sections are elements of
\begin{equation}\label{twoinds}
  \operatorname{Ind}_{B\cap K}^K\delta \ \ \cong  \ \  V_{\nu,\d}\,,
\end{equation}
 which is a fixed space of functions on $K$ to which we will apply operations with varying $\nu$.

\section{Limiting structures on intertwining operators}\label{sec:intertwiners}

This section contains the core concepts of our argument, which we first summarize as follows.
Unitary quotients $\pi$ of a dominant minimal principal series representation $V_{\nu,\d}$ (\ref{princseries}) of $G=E_{8(8)}$  must preserve a positive-definite Hermitian form, which is given in terms of the long intertwining operator in (\ref{hermitianform}) \cite[\S4]{Vogan}. We will work with Langlands' normalization of intertwining operators $R(w_{\text{long}},\nu,\delta)$ \cite[Appendix II]{langlands}.  This normalization is designed so that it acts as simply as possible on the lowest $K$-types of principal series, up to some choices of signs.  We will establish Theorem~\ref{mainthm} by showing that if $\langle \nu,\alpha^\vee\rangle > 1$ for some simple coroot $\alpha^\vee$, then $\pi$ is not unitary.  This will in turn be shown by demonstrating that $R(w_{\text{long}},\nu,\delta)$ is indefinite on some $K$-type of $\pi$.  It then suffices to prove the indefiniteness for   some $K$-type of the full principal series representation $V_{\nu,\d}$, since the indefiniteness must persist to $\pi$, as
$V_{\nu,\d}$ has a unique lowest $K$-type.  Finally, this last statement is shown using the theory of limiting intertwining operators.

\subsection{Intertwining integrals}

We record some basic properties of the standard intertwining operators mapping $V_{\nu,\delta}$ to $V_{w\nu,w\delta}$, $w\in W$, referring to \cite[\S9]{Hundley-Miller} as a recent reference aligned with the presentation here.   The standard intertwining operator
\begin{equation}\label{mspaces}
M(w,\nu,\d)\colon V_{\nu,\d} \ \ \rightarrow \ \  V_{w\nu,w\delta}
\end{equation}
 is defined as an integral that makes sense for $w\in N(T)$, the normalizer of $T$:
\begin{equation}\label{standardint}
  [M(w,\nu,\delta)f_\nu](g) \ \ = \ \
 \int_{N\cap wN_{-}w^{-1}}f_\nu(w^{-1}ug)\, du\,, \ \ \ \ f_\nu \, \in\, V_{\nu,\d}^\infty\,,
\end{equation}
where  --- at least to initially ensure absolute convergence of the integral --- $\nu$ satisfies the constraint that $\langle \nu,\alpha^\vee\rangle$ has positive real part for any $\alpha\in\Delta_+\cap w^{-1}\Delta_{-}$.  The integral in fact meromorphically continues to all $\nu\in \frak h^*_\C$.  Definition (\ref{standardint}) of course makes sense for arbitrary $w\in W$, using the identification with elements of $N(T)\cap G$ described in \secref{sec:background}.  Note that had we taken a different representative of $w$ inside $N(T)$, the value of (\ref{standardint}) would only change by a scalar coming from the inducing character.

Langlands introduced the scalar normalizing factors
\begin{equation}\label{normalizingfactors}
  \sigma(w,\nu,\delta) \ \ = \ \ \prod_{\srel{\alpha\in\Delta_+}{w\alpha\in\Delta_{-}}} \sigma(\langle \nu,\alpha^\vee\rangle,\delta\circ \alpha^\vee)\,,
\end{equation}
where
\begin{equation}\label{cfunction}
  \sigma(s,\xi) \ \ = \ \ \left\{
                        \begin{array}{ll}
                         \frac{1}{\sqrt{\pi}} \frac{\Gamma\(\smallf{s+1}{2}\)}{\Gamma\(\smallf{s}{2}\)}, & \xi\hbox{ trivial;} \\
\\
                         \frac{i}{\sqrt{\pi}} \frac{\Gamma\(\smallf{s+2}{2}\)}{\Gamma\(\smallf{s+1}{2}\)}, & \xi\hbox{ nontrivial.}
                        \end{array}
\right.
\end{equation}
The normalized intertwining operators
\begin{equation}\label{normalizedintertwiners}
  R(w,\nu,\delta) \ \ := \ \ \sigma(w,\nu,\delta) M(w,\nu,\delta)
\end{equation}
enjoy several important properties, for example:
\begin{enumerate}
  \item they are holomorphic in the range $\{\nu\in \frak h_\C^* | \Re\! \langle \nu, \alpha^\vee\rangle > -1,\forall \alpha\in\Delta_+\cap w^{-1}\Delta_{-}\}$;
  \item they are nonzero  scalar multiples of $M(w,\nu,\d)$ for $\{\nu\in \frak h_\C^* | \Re\! \langle \nu, \alpha^\vee\rangle > 0,\forall \alpha\in\Delta_+\cap w^{-1}\Delta_{-}\}$;
  \item they obey the composition rule
\begin{equation}\label{multrule}
  R(w_1w_2,\nu,\delta) \  =  \ R(w_1,w_2\nu,w_2\delta)\circ R(w_2,\nu,\delta)
\end{equation}
for any pair of  elements $w_1,w_2$ in the normalizer of $T$; and
  \item if $\delta$ is trivial, they act trivially on the spherical $K$-type (in particular, on flat sections whose restriction to $K$ is constant).
\end{enumerate}
Note again that there is a subtlety in applying (\ref{multrule}) when $w_1$ and $w_2$ are taken to be the fixed representatives of elements of $W$ chosen in \secref{sec:background}:~unless $\ell(w_1w_2)=\ell(w_1)+\ell(w_2)$, the product of the representatives is not guaranteed to be the representative of the product.
Nevertheless, if $w$ is written as a word of minimal length in simple Weyl reflections, the multiplication rule (\ref{multrule}) reduces the computation of $R(w,\nu,\d)$ to the case that $w$ is a simple Weyl reflection.  In turn, the evaluation of (\ref{standardint}) for simple Weyl reflections reduces to an $SL(2)$ calculation that will be carried out in  detail in the next two subsections.

Recall that for groups such as  $G=E_{8(8)}$ that have no nontrivial Dynkin diagram automorphisms, $w_\text{long}\nu=-\nu$.
Up to scalar multiples, the invariant Hermitian form on pairs $f,f'\in V_{\nu,\delta}$ is conveniently given through the normalized intertwining operators $R(w_{\text{long}},\nu,\delta):V_{\nu,\delta}\rightarrow V_{-\nu,\delta}$ by the formula
\begin{equation}\label{hermitianform}
  H(f,f') \ \ = \ \ \langle f,\overline{R(w_{\text{long}},\nu,\delta)f'}\rangle,
\end{equation}
where $\langle v,v'\rangle$ is the pairing on $V_{\nu,\delta}\times V_{-\nu,\delta}$ given by integrating the product of $v$ and $v'$ over $B\backslash G$ (the shift by $\rho$ in (\ref{princseries}) makes this integral well-defined).

\subsection{Review of intertwiners for $SL(2,\R)$}\label{sec:sub:SL2intertwiners}

In this subsection we temporarily switch our focus to  $SL(2,\R)$, for which there is a single nontrivial Weyl element $w=w_{\text{long}}$ and (\ref{standardint}) becomes
\begin{equation}\label{standardintSL2}
  [M(w,\nu,\delta)f_{\nu}](g) \ \ = \ \ \int_{x\in\R} f_\nu\(\ttwo{0}{-1}{1}{x} g\)\,dx\,.
\end{equation}
The $K$-types in $V_{\nu,\delta}$ are constrained by the parity of $\delta$:~the $K$-isotypic subspaces are spanned by flat sections $f_\nu$ whose restrictions to $K=SO(2)$ are characters of the form
\begin{equation}\label{Kisotypic}
 \ttwo{d}{-c}{c}{d} \ \mapsto \ (d+ic)^n,
\end{equation}
with $n$ even if $\delta$ is trivial and $n$ odd if $\delta$ is nontrivial.  Since intertwining operators preserve these 1-dimensional $K$-isotypic spaces, the integral (\ref{standardintSL2})   must be a scalar multiple of $f_\nu$, and to compute this scalar it suffices to specialize $g=e$.  Using the Iwasawa matrix decomposition
\begin{equation}\label{SL2matrix}
  \ttwo{0}{-1}{1}{x}  \ \ = \ \ \ttwo{(x^2+1)^{-1/2}}{\star}{0}{(x^2+1)^{1/2}} \ttwo{\frac{x}{\sqrt{x^2+1}}}{-\frac{1}{\sqrt{x^2+1}}}{
\frac{1}{\sqrt{x^2+1}}}{\frac{x}{\sqrt{x^2+1}}},
\end{equation}
 the integral (\ref{standardintSL2}) becomes
\begin{equation}\label{standardintSL2explicated}
\aligned
  [M(w,\nu,\delta)f_{\nu}](e) \ \ & = \ \ \int_{\R} f_\nu\(\ttwo{(x^2+1)^{-1/2}}{\star}{0}{(x^2+1)^{1/2}} \ttwo{\frac{x}{\sqrt{x^2+1}}}{-\frac{1}{\sqrt{x^2+1}}}{
\frac{1}{\sqrt{x^2+1}}}{\frac{x}{\sqrt{x^2+1}}}\)\,dx\\
& = \ \  \int_{\R}(x^2+1)^{-(\langle \nu,\a^\vee\rangle+1)/2}\, f_\nu\( \ttwo{\frac{x}{\sqrt{x^2+1}}}{-\frac{1}{\sqrt{x^2+1}}}{
\frac{1}{\sqrt{x^2+1}}}{\frac{x}{\sqrt{x^2+1}}}\) \,dx\\& = \ \  \int_{\R}(x^2+1)^{-(\langle \nu,\a^\vee\rangle+1)/2}\, \(\f{x+i}{|x+i|}\)^n \,dx\\
& = \ \ \frac{r(\langle \nu,\a^\vee\rangle,n)}{ \sigma(\langle \nu,\a^\vee\rangle,\delta)}\,,
\endaligned
\end{equation}
where
\begin{equation}\label{ras}
r(s,n) \ \ := \ \
\begin{cases}
 \displaystyle\prod_{\substack{k=1 \\ k\text{ odd}}}^{|n|-1}
\smallf{k-s}{k+s},
& \text{if $n$ is even $\Longleftrightarrow$ $\delta$ is trivial},\\[1.2em]
 \sgn(-n)\displaystyle\prod_{\substack{k=2 \\ k\text{ even}}}^{|n|-1}
\smallf{k - s}{k+s},
& \text{if $n$ is odd $\Longleftrightarrow$  $\delta$ is nontrivial}.
\end{cases}
\end{equation}
Thus the normalized intertwining operator $R(w,\nu,\delta)$ (\ref{normalizedintertwiners}) acts on the $K$-isotypic space for (\ref{Kisotypic}) by the scalar $r(\langle \nu,\a^\vee\rangle,n)$.
Finally, note that since the diagonal entries of the upper triangular matrix in (\ref{SL2matrix}) are positive, the calculation in (\ref{standardintSL2explicated}) uses nothing directly about $\delta$:~its influence is only felt in determining the parity of $n$ in (\ref{Kisotypic}).

\subsection{Calculation of simple intertwiners}\label{sec:sub:calculation}

Having given details in the case of $SL(2,\R)$, we now return to the general setting  and
resume the calculation of the $R(w_{\beta},\cdot,\cdot)$, $\beta\in\Sigma$, which are the factors of $R(w_{\text{long}},\cdot,\cdot)$ through the multiplication rule (\ref{multrule}) (written out in  (\ref{comprulelong})).  Intertwining operators preserve $K$-types, and hence the action can be computed  one $K$-type at a time.  This calculation is well-known, and similar approaches are taken in \cite{Vogan,Zhang}.

First we describe the flat sections for an arbitrary $K$-type $(\tau,U)$, as elements of $C^\infty(K)$.  Let $(\tau',U')$ be the  representation dual to $(\tau,U)$, with pairing $\langle \cdot,\cdot\rangle:U'\times U\rightarrow \C$.  By the Peter-Weyl theorem, the $\tau$-isotypic subspace of $C^\infty(K)$ is isomorphic to $\tau'\otimes \tau$, realized concretely for pure tensors by the matrix coefficients
\begin{equation}\label{matrixcoefficients}
  F_{u',u}(k) \ \ = \ \ \langle u',\tau(k)u\rangle, \qquad u'\in U', \,u\in U.
\end{equation}
 Flat sections $f_\nu$ in (\ref{twoinds})  transform on the left according to the character $\delta$; in terms of (\ref{matrixcoefficients}), this condition is that
\begin{equation}\label{taum}
\tau'(m) u' \ \ = \ \ \delta(m)\,,\qquad \forall m\,\in\,M
\end{equation}
(since both $m$ and $\delta$ have order at most 2,  inverses are not needed to describe this equivariance).

With our choice of Weyl group representatives,
definition (\ref{standardint}) specializes for simple intertwining operators to
\begin{equation}\label{simpleintforbeta1}
  [M(w_\b,\nu,\d)f_\nu](g) \ \  = \ \ \int_{\R} f_\nu\(\Phi_\b(\ttwo{0}{-1}1x)g\)dx\,.
\end{equation}
As both $f_\nu$ and $[M(w_\b,\nu,\d)f_\nu](g)$ transform by (different) inducing characters of $B$ on the left, it suffices to take $g=k\in K$ by the Iwasawa decomposition.
The calculation of \secref{sec:sub:SL2intertwiners} now  goes through largely intact in this more general setting through the homomorphism $\Phi_\beta$.   The upper triangular matrix on the right-hand side of (\ref{SL2matrix}) again gives a power  through $\Phi_\beta$:
\begin{equation}\label{simpleintforbeta2}
  [M(w_\b,\nu,\d)f_\nu](k) \ = \ \int_{\R} (x^2+1)^{-(\langle \nu,\beta^\vee\rangle+1)/2} f_\nu\(\!\Phi_\b\!\(\!\ttwo{\frac{x}{\sqrt{x^2+1}}}{-\frac{1}{\sqrt{x^2+1}}}{
\frac{1}{\sqrt{x^2+1}}}{\frac{x}{\sqrt{x^2+1}}} \! \) k \! \)dx.
\end{equation}
In order to simplify further, we diagonalize $U'$ by the Lie algebra element $Z_\beta=X_\b-X_{-\b}$.  (Note that the differential action of $d\tau'(Z_\beta)$ need not preserve the condition   (\ref{taum}), nor should it:~(\ref{standardint}) maps $V_{\nu,\delta}$ to $V_{w\nu,w\delta}$.)
  In particular we assume, as we may through linear combinations, that $f_\nu$'s restriction to $K$ is of the form $F_{u',u}$ from (\ref{matrixcoefficients}), with $d\tau'(Z_\b)u'=inu'$,  $n$  necessarily an integer.  Consequently, $f_\nu$ transforms on the left under $\Phi_\b(SO(2))$ by the character (\ref{Kisotypic}), and (\ref{simpleintforbeta2}) becomes
\begin{equation}\label{simpleintforbeta3}
\aligned
  [M(w_\b,\nu,\d)f_\nu](k) \ \  & = \ \ \int_{\R} (x^2+1)^{-(\langle \nu,\beta^\vee\rangle+1)/2} \(\f{x+i}{|x+i|}\)^n F_{u',u}(k)\, dx \\
 & = \ \ \frac{r(\langle \nu,\b^\vee\rangle,n)}{ \sigma(\langle \nu,\b^\vee\rangle,\delta)}\,F_{u',u}(k)\,,
\endaligned
\end{equation}
as in (\ref{standardintSL2explicated}).
That is, $R(w_\beta,\nu,\delta)$ acts by the scalar $r(\langle \nu,\b^\vee\rangle,n)$ on the flat section for the matrix coefficient  $F_{u',u}$ on $K$, whenever $d\tau(Z_\b)u'=inu'$.

Of course, the equivariance of the intertwining operators allows us to  replace $k$ with $e$ in (\ref{simpleintforbeta3}), without losing any generality.  This is equivalent to replacing $u$ with a translate;  the factor in formula (\ref{simpleintforbeta3}) depends on $u'$ but not $u$.  Accordingly, taking $u$ to be a fixed, nonzero vector allows us to cut down the matrix size of the Hermitian form's restrictions to $K$-types, which can therefore be completely described by using the $M$-equivariant $u'$ in (\ref{matrixcoefficients})-(\ref{taum}).

This computation of the simple intertwining operators $R(w_\beta,\nu,\delta)$ shows they are diagonal on an eigenbasis for $d\tau'(Z_\beta)$.
The general intertwining operator is now a composition of these simple intertwining operators, though with different values of $\nu$ and  potentially different eigenbases.  Note that once again the character $\delta$ plays no direct role in the computation, aside from influencing which integers occur as  eigenvalues of $i Z_\beta$.  

\subsection{Large-$s$ limit of simple intertwiners}\label{sub:sec:largelimit}

In order to later establish the requisite asymptotics of the Hermitian form restricted to certain $K$-types,   we compute those of the simple normalized intertwining operators $R(w_\b,\nu,\d)$ in the large-$\langle \nu,\b^\vee\rangle$ limit.
First,
\begin{equation}\label{rasymptotics}
  \lim_{s\rightarrow\infty}r(s,n) \ = \ \left\{
                                          \begin{array}{ll}
                                            1, & n\equiv 0\text{~or~}3\pmod 4 \\
                                            -1, & n\equiv 1\text{~or~}2\pmod 4.
                                          \end{array}
                                        \right.
\end{equation}
This follows from (\ref{ras}), and can also be seen from the localization of the integral in (\ref{standardintSL2explicated}) around $x=0$ by applying Lemma~\ref{lem:asymptconcentration} below.

Next,  applying Stirling's formula to (\ref{cfunction}) shows the $s\rightarrow\infty$ asymptotics
\begin{equation}\label{sterling}
  \sigma(s,\xi) \  \sim  \ \left\{
                        \begin{array}{ll}
                        \sqrt{\frac{s}{2\pi}}, & \xi\hbox{ trivial,} \\
\\
                        i\sqrt{\frac{s}{2\pi}}, & \xi\hbox{ nontrivial,}
                        \end{array}
\right.
\end{equation}
which then gives full asymptotic control of the quotient in (\ref{standardintSL2explicated}) as $\langle \nu,\a^\vee\rangle\rightarrow\infty$.  We shall use the following Lemma to deduce  more general asymptotics of the simple intertwining integral operators from (\ref{simpleintforbeta1}):

\begin{lem}\label{lem:asymptconcentration}
Let $\phi(x)$ denote a fixed, continuously differentiable,  bounded function on the real line.  Then
\begin{equation}\label{asymptlittleo1}
  \int_{\R}(x^2+1)^{-(s+1)/2}\,\phi(x)\,dx  \ - \  \frac{\phi(0)}{\sigma(s,\xi_\text{triv})}  \ \ = \ \ o\(\frac{1}{\sigma(s,\xi_\text{triv})}\)
\end{equation}
as $s\rightarrow\infty$ along the positive reals.
\end{lem}

\noindent The intuition behind this lemma is that for large $s$, the integrand is mainly influenced by small $x$.  The exact rate of decay of contributions for nonzero $x$ is determined by $\sigma(s,\xi_\text{triv})$.  Such an analytic estimate on the integrand's contributions is crucial for the application to limiting structures on intertwining operators, and would be much more difficult in higher dimensions without the factorization of general intertwining operators into simple ones provided by (\ref{multrule}).
\begin{proof}
First, explicit calculation compared with  (\ref{cfunction}) shows that
\begin{equation}\label{evalintegral}
  \int_\R(x^2+1)^{-(s+1)/2}\,dx \ \ = \ \ \frac{1}{\sigma(s,\xi_\text{triv})} \,, \ \ \Re{s}>0.
\end{equation}
In order to show the assertion, bound the left-hand side of (\ref{asymptlittleo1}) for $s>0$ using
\begin{multline}\label{asymptlittleo2}
  \left| \int_{\R}(x^2+1)^{-(s+1)/2}\,(\phi(x)-\phi(0))\,dx \right| \ \le   \\   \int_{|x|< 1}(x^2+1)^{-(s+1)/2}\,|\phi(x)-\phi(0)|\,dx \ + \ \int_{|x|\ge 1}(x^2+1)^{-(s+1)/2}\,|\phi(x)-\phi(0)|\,dx\,.
\end{multline}
Since $\phi(x)-\phi(0)=\int_0^x \phi'(t)dt$, the first integrand is bounded by a multiple $|x|(x^2+1)^{-(s+1)/2}$, whose integral over the entire real line is $\f{2}{s-1}$.  The second integrand is $O(x^{-s-1})$, and therefore its integral is also $O(s\i)$, which is  $o(\sigma(s,\xi_\text{triv}))$ by (\ref{sterling}).
\end{proof}

As a consequence of the Lemma and (\ref{simpleintforbeta1}), for any flat section $f_\nu$ in (\ref{twoinds}) one has the large real $\langle \nu,\beta^\vee\rangle$-asymptotics 
\begin{equation}\label{lemmaconsequenceM}
  [M(w_\b,\nu,\d)f_\nu](e) \ \ = \  \ \frac{1}{\sigma(\langle \nu,\b^\vee\rangle,\xi_\text{triv})}\,f_\bullet(w_\beta\i) \ + \  o\(\frac{1}{\sigma(s,\xi_\text{triv})}\),
\end{equation}
where $f_\bullet (w_\beta\i)$ denotes the common value of the $f_\nu(w_\beta\i)$.
Note that the denominator in (\ref{lemmaconsequenceM}) is independent of $\delta$. The normalized intertwining operators therefore have the large-$\langle \nu,\beta^\vee\rangle$ limits
\begin{equation}\label{lemmaconsequenceR}
 \lim_{  \langle \nu,\beta^\vee\rangle\rightarrow\infty} [R(w_\b,\nu,\d)f_\nu](e) \ \ = \  \
\left\{
  \begin{array}{ll}
    f(w_\beta\i), & \d\circ \beta^\vee\text{~trivial}, \\
    if(w_\beta\i), & \hbox{otherwise}
  \end{array}
\right.
\end{equation}
because of (\ref{sterling}).
This localization of intertwining operators at $\infty$ is the core analytic ingredient in this paper.  As an example, for $SL(2,\R)$ it shows that the limits for $f_\nu$ coming from the characters (\ref{Kisotypic}) are explicit roots of unity; applied to the case of $n=2$, where the only singularities are at $\nu=\pm \rho$, we recover the famous fact that the spherical complementary series  $\nu=t\rho$ ends at  $t=1$.

\section{Proof of the  Limiting~\thmref{thm:limit}}\label{sec:limittheorem}

Let $\Lambda$ denote the left translation operator on functions in $C^\infty(K)$.  As it does for any $w\in W$, the composition rule (\ref{multrule}) allows us to write $R(w,\nu,\delta)$ for $w=w_\text{long}$ as
\begin{equation}\label{comprulelong}
\aligned
  R(w_{\text{long}},\nu,\delta) \ \ & = \ \ R_1(\nu,\delta)\circ R_2(\nu,\delta)\circ\cdots\circ R_{\ell}(\nu,\delta)\,,\\
\text{where} \ \ R_j(\nu,\delta) \ \ & := \ \ R(w_{\b_j},w_{\b_{j+1}}\cdots w_{\b_\ell}\nu,w_{\b_{j+1}}\cdots w_{\b_\ell} \d)
\,,
\endaligned
\end{equation}
for any reduced factorization $w_{\beta_1}w_{\beta_2}\cdots w_{\b_\ell}$ of $w$, $\ell=\ell(w)$.  It is manifest from  this and (\ref{simpleintforbeta2}) that  $R_j(\nu,\delta)$'s $\nu$-dependence depends only on the value of
\begin{equation}\label{factorsthrough}
\langle w_{\b_{j+1}}\cdots w_{\b_\ell}\nu,\beta_j^\vee\rangle \ \ = \ \
\langle \nu,w_{\b_\ell}\cdots w_{\b_{j+1}} \beta_j^\vee\rangle\,,
\end{equation}
where $w_{\b_\ell}\cdots w_{\b_{j+1}} \beta_j^\vee$ is a positive coroot.
It follows from (\ref{lemmaconsequenceR}) that
\begin{equation}\label{vanishingimplication}
\aligned
 \langle \varpi_\beta,w_{\b_\ell}\cdots w_{\b_{j+1}} \beta_j^\vee \rangle = 0 \ \ & \Longrightarrow \ \
R_j(\nu+s\varpi_\beta,\d)=R_j(\nu,\d) \text{~is constant in $s$,}\\
 \langle \varpi_\beta,w_{\b_\ell}\cdots w_{\b_{j+1}} \beta_j^\vee \rangle > 0 \ \ & \Longrightarrow \ \
 \lim_{s\rightarrow\infty} R_j(\nu+s\varpi_\beta, \d) = \epsilon_{w_{\beta_{j+1}}\cdots w_{\b_\ell}\d,\beta_j} \,\Lambda(w_{\beta_j}),
\endaligned
\end{equation}
where
\begin{equation}\label{epsilonfactor}
  \e_{\xi,\a} \ \ = \ \ \left\{
                                \begin{array}{ll}
                                  1, & \xi\circ\a^\vee\ \text{trivial,} \\
                                  i, & \hbox{otherwise,}
                                \end{array}
                              \right.
\end{equation}
takes into account  the normalization difference between the two cases of (\ref{lemmaconsequenceR}).

We now apply the decomposition (\ref{Kostant}).  Let $\ell'=\ell(w')$ and $\ell''=\ell(w_{\text{long},\beta})$,  writing $w'=w_{\beta_1}\cdots w_{\beta_\ell'}$ and $w_{\text{long},\beta}=w_{\beta_{\ell'+1}}\cdots w_{\beta_\ell}$ as reduced products of simple reflections.  As remarked after (\ref{Kostant}), $w'w_{\text{long},\beta}=w_{\beta_1}\cdots w_{\beta_\ell}$ is a reduced word for $w_{\text{long}}$.
 Then for all $j>\ell'$, $w_{\b_\ell}\cdots w_{\b_{j+1}}\beta_j^\vee$ is a positive coroot for $L_\beta$, which is the first case in (\ref{vanishingimplication}).  Moreover, each of the $\ell''$ positive coroots for $L_\beta$ occurs this way.  At the same time, the coroots $w_{\b_\ell}\cdots w_{\b_j+1}\b_j^\vee$ for $1\le j \le \ell=\ell'+\ell''$ exhaust all of the positive coroots for $G$.  It follows that if $1\le j \le \ell'$, the second case of (\ref{vanishingimplication}) must hold.

 Assume now that $\langle \nu,\beta^\vee\rangle =0$ as in the statement of~\thmref{thm:limit}.  It follows that
\begin{equation}\label{limitwithLambda}
\aligned
\lim_{s\rightarrow\infty} R(w_{\text{long}},\nu+s\varpi_\beta,\delta)  \ & = \  \e\,
\Lambda(w_{\beta_1}) \!\circ \!\cdots\! \circ \!\Lambda(w_{\beta_{\ell'}})\!\circ\!
R_{\ell'+1}(\nu,\d)\!\circ\! \cdots \circ R_{\ell}(\nu,\d) \\
& = \  \e\,\Lambda(w_{\beta_1}\cdots w_{\beta_{\ell'}})\circ R(w_{\text{long},\b},\nu,\d)\\
& =  \ \e\,\Lambda(w')\circ R(w_{\text{long},\b},\nu,\d)\,,
\endaligned
\end{equation}
where $\e=\prod_{j=1}^{\ell'} \epsilon_{w_{\beta_{j+1}}\cdots w_{\b_\ell}\d,\beta_j}$.
The Theorem now follows. \bx

\section{Proving the FPP inequality}\label{sec:FPPproof}

Barbasch \cite{petite} introduced the notion of {\it petite} $K$-types in his study of the spherical unitary dual for split classical groups:~irreducible representations $(\tau,U)$  of $K$ for which the spectrum of $d\tau(iZ_\beta)$ on $U$ is contained in $\Z\cap[-2,2]$.  This notion has been considered as well in the non-spherical situation \cite{jeringu,annegret}.
  The spectral bound simplifies the factors of $r(s,n)$ in (\ref{ras}):
\begin{equation}\label{smallras}
 r(s,0) \ = \ r(s,-1) \ = \ -r(s,1) \ = \ 1,  \text{~~and~}r(s,\pm 2) \ = \ \frac{1-s}{1+s}\,.
\end{equation}
 Their most important feature here is that
\begin{equation}\label{rasdoesntchangesign}
  \text{ for $-2\le n \le 2$,~} r(s,n) \ \text{has no singularities or sign changes in }1<s<\infty,
\end{equation}
and that this common sign is given by the limiting formula (\ref{rasymptotics}).
Since the restriction of the Hermitian form $R(w,\nu,\delta)$ to a $K$-type depends continuously on $\nu$, these properties show that the signature restricted to a petite $K$-type is constant on a   large region in $\nu$, and can be computed from limiting values (where the expressions from the last two sections are simpler).

More specifically, the only possible singularities of $R(w_{\text{long}},\nu,\delta)$  on petite $K$-types for   dominant $\nu$ are on hyperplanes of the form $\langle \nu,\alpha^\vee\rangle =1$ for some $\alpha \in \Delta_+$.  Fix $\beta\in\Sigma$.  Suppose that $\nu$ is written as a linear combination of fundamental weights $\sum c_\varpi \varpi$, with the coefficient $z:=c_{\varpi_\beta}=\langle \nu,\beta^\vee\rangle$ of the fundamental weight $\varpi_\beta$ dual to $\beta^\vee$ allowed to vary, but all other $c_\varpi$ fixed.  Let $\gamma^\vee$ be an arbitrary positive coroot, expressed as an integral combination $\sum_{\alpha\in\Sigma}c_\a \a^\vee$ of  positive simple coroots.  Then
\begin{equation}\label{pairingwithmovingnu}
  \langle \nu,\gamma^\vee \rangle  \ \ = \ \ z\, c_\beta \ + \ \left\langle \sum_{\varpi\neq\varpi_\b}c_\varpi \varpi ,\gamma^\vee- c_\b \b^\vee \right\rangle
\end{equation}
is an affine function in $z=\langle \nu,\beta^\vee\rangle$.  The second term on the right-hand side is nonnegative for $\nu$ dominant, as is $c_\beta$.  Let us now consider a deformation of $z$ over the interval $(1,\infty)$.  If $c_\beta=0$, then (\ref{pairingwithmovingnu}) remains constant.  Otherwise, the integer $c_\beta$ must be at least 1, and (\ref{pairingwithmovingnu}) is always strictly greater than 1 for any $1<z<\infty$.  Using (\ref{rasdoesntchangesign}) we conclude that the signature at $\nu$ of this form on a petite $K$-type is unchanged as $z$ grows larger and larger.  Recall from \thmref{thm:limit} that the limiting value of $R(w,\nu,\delta)$ exists in the large-$z$ limit.  If this $z\rightarrow \infty$ limit is indefinite, then $V_{\nu,\delta}$ cannot be unitary if $\nu$ is dominant and $z=\langle \nu,\beta^\vee\rangle>1$.  Furthermore, in situations such as ours where  the principal series has a unique lowest $K$-type, that indefiniteness must persist to the quotient of $V_{\nu,\delta}$.

To prove Theorem~\ref{mainthm}, it remains to show that for each character $\d$ and each simple root $\beta$, there exists a petite $K$-type $\tau$ on which the limiting value of $R(w_\text{long},\nu,\delta)$ in the $z=\langle \nu,\beta^\vee\rangle\rightarrow \infty$ limit is indefinite.  This is adequate for showing non-unitarity, since the Hermitian form  is known to be positive definite on the unique lowest $K$-type, as can be checked by directly computing  $R(w_\text{long},0,\delta)$.

The verification of the limiting indefiniteness is accomplished by a case-by-case argument, in which it is shown that the limiting value of $R(w_\text{long},\nu,\delta)$ has an explicit eigenvector with eigenvalue $-1$.  We successfully carried out this strategy for all split real forms of exceptional groups, and  present details for $E_{8(8)}$ below.  We note that the proof in \cite{DM-B} also proves non-unitarity by finding an eigenvector with eigenvalue $-1$, though it is unclear if there is  any connection with the present argument.

\subsection{Implementation issues}

Our computations use explicit models of finite-dimensional representations of $K$, especially in terms of a matrix realization of the Lie group $G$ in which the $M$-action can be explicitly calculated.  We used the matrix realizations of exceptional groups $G$ due to Ross Lawther, which were publicly made available in \cite{Hundley-Miller}.

On the other hand, explicit models of arbitrary  representations of $K$  were not available in the literature in a useful form.   Zhuohui Zhang recently found
a beautiful construction using Gelfand-Tsetlin bases, and graciously made his code \cite{Zhangrepo} available to us;  we were unable to scale up other existing approaches to adequately perform our necessary calculations.
  Deriving efficiently computable large-dimensional models for $K$-types is nontrivial and delicate.  However, it is important to note that the application here only requires verification of the correctness of the model representation (but not its derivation), which is straightforward to check using generators and relations.
For this reason we will instead focus on the implementation aspects specific to computing intertwining operators.

Symbolic, exact calculation becomes cumbersome in high dimensions (our most useful $K$-type has dimension 1920).  To circumvent this bottleneck we instead used numerical methods in such a way that the results could then be rigorously checked symbolically.
  First, we used numerical eigenvalue packages as a tool to quickly approximate the spectrum of the $Z_\alpha$-action, which on any particular $K$-type is a matrix action.  Though the eigenvalues are integers, symbolic computation is prohibitive.  This serves as a useful check.  It is then a much simpler matter to rigorously and symbolically diagonalize  this action for each rounded eigenvalue, thereby verifying the numerical findings by checking that  dimensions of the eigenspaces add up to the total dimension of the representation space.

Second, it is difficult to symbolically compute the change-of-basis matrices between the diagonalizations of the $Z_\beta$-actions for different $\beta$.  Computing numerical approximations, on the other hand, is very easy, but not rigorous.  We blended the advantages of both approaches as follows.  Making an ansatz that the entries of the change-of-basis matrices should be algebraic numbers, we applied the LLL algorithm   to find excellent algebraic number approximations to the numerically-calculated values of the entries of the matrices.  With such a candidate guess  for these exact, symbolic values in hand, we then performed the computationally simpler task of verifying the algebraic approximations to the change-of-basis matrices were in fact correct.

Finally, speed issues were circumvented in most of the calculations by finding an appropriate conjugate of Zhang's representations for which most calculations could be performed using integer arithmetic, rather than relying on symbolic computation with algebraic numbers.

\subsection{Particular details for $E_{8(8)}$}

Recall that the maximal compact subgroup $K$ of $E_{8(8)}$ is a quotient of $Spin(16)$ by an element of order 2 (but  is not $SO(16)$).  The lowest-dimensional $K$-types are listed in Table~\ref{E8table}, all of which are petite.  Recall the Cartan decomposition ${\frak g}={\frak k}\oplus {\frak p}$, where $\frak k=Lie(K)$ and
\begin{equation}\label{frakpdef}
  \frak p \ \ = \ \ (\frak h \otimes \C)\,\oplus\,\bigoplus_{\a\in\Delta_+}\C(X_\a+X_{-\a}).
\end{equation}
 The adjoint action of $K$ on $\frak g$ preserves both $\frak k$ and $\frak p$, which in this particular case gives natural, internal realizations of the 120- and 128-dimensional representations of $K$.

There are 256 possible principal series representations (\ref{princseries}), corresponding to the 256 possible characters $\delta$.  The intertwining operators  (\ref{standardint}) group them into three   families, with related values of $\nu$:
\begin{itemize}
  \item the spherical principal series $V_{\nu,\delta_{\text{triv}}}$;
  \item the 120 $V_{\nu,\delta}$ for which $\delta$ is in the Weyl orbit of the character $\delta_{\varpi_8}$ associated to the fundamental weight $\varpi_8$; and
  \item the remaining 135 in the Weyl orbit of the character $\delta_{\varpi_2}$ associated to $\varpi_2$.
\end{itemize}
 Each principal series has a unique lowest $K$-type (LKT) occurring with multiplicity 1, which for these three families respectively are:~the trivial representation, the 128-dimensional representation $\frak p$, and the 135-dimensional representation.  The sums of the multiplicities  across all 256 possible $V_{\nu,\delta}$ of a $K$-type $\tau$ must equal the dimension of $\tau$; from this we see that $\frak p$ must also occur with multiplicity  exactly 8 inside the spherical principal series (see~\secref{sec:spherical}), and that none of these LKTs otherwise appear in the other families.

\begin{table}
\begin{center}
\begin{tabular}{c|c|c}
Highest weight & dimension & comments \\
\hline
0 & 1 & trivial, LKT \\
$\varpi_2$ & 120 & $\frak k$, LKT\\
$\varpi_8$ & 128 & $\frak p$, half-spin, useful \\
$2\varpi_1$ & 135 & LKT \\
$\varpi_4$ & 1820 &  not useful\\
$\varpi_1+\varpi_7$ & 1920 & vector-spinor, very useful \\
\end{tabular}
\end{center}
\caption{The lowest-dimensional $K$-types occurring in principal series for $E_{8(8)}$, all of which listed here are petite.  The indexing for the highest weights $\varpi_7$ and $\varpi_8$ depends on particular choice of coordinates for ${\frak k}=Lie(K)$ inside ${\frak g}=Lie(G)$ (they might be switched in other coordinates, owing to $K$'s Dynkin diagram automorphism).  ``LKT'' signifies lowest $K$-types of some principal series representations, on which the value of $R(w_\text{long},\nu,\delta)$  can be thought of as a normalization (LKTs themselves are inadequate to prove non-unitarity).
\label{E8table}}
\end{table}

Barbasch proved the FPP inequality for spherical representations using the 128-dimensional $K$-type $\frak p$.  In the next section we give a proof of the spherical case as a direct application of \thmref{thm:limit}.
This  case of $\delta_{\text{triv}}$ trivial aside, there are $2040=(2^8-1)\times 8$ different pairs of nontrivial characters $\delta$ and perturbation directions $\beta$.  The FPP inequality for $E_{8(8)}$ will be proved by showing each of these 2040 limits of intertwining operators has a $-1$ eigenvalue on some petite $K$-type.

\vspace{.2cm}

{\bf Applications of $\frak p$.} Characters $\delta$ correspond to elements of $\Lambda_{\text{wt}}/2\Lambda_{\text{wt}}$, where $\Lambda_{\text{wt}}$ is the weight lattice. Suppose $\delta$ is one of the 120 characters in the Weyl orbit of $\d_{\varpi_8}$, so that the $K$-type $\frak p$ occurs with multiplicity one in $V_{\nu,\d}$.  Noting that $w_{\text{long}}$ acts as the identity on the characters of $M$ but as negation on  $\Lambda_{\text{wt}}$, write a weight corresponding to the character $\delta$ as $w\varpi_8$ for some $w\in W$,  replacing $w$ by $-w$ if necessary to ensure that $w\a_\text{high}\in \D_+$.  Then the computation of intertwining operators outlined above shows that  $R(w_\text{long},\nu,\delta)$ acts by the scalar $\frac{1-\langle \nu,w\a_\text{high}^\vee\rangle}{1+\langle \nu,w\a_\text{high}^\vee\rangle}$ on the $\frak p$-isotypic space of $V_{\nu,\delta}$, which occurs with multiplicity one.

Suppose that $\beta$ is a simple root for which $w\a_\text{high}^\vee$ has a nontrivial coefficient of $\beta^\vee$, so that  $\langle \nu,w\a_\text{high}^\vee\rangle$ grows to $\infty$ in the large-$\langle \nu,\beta^\vee\rangle$ limit.  It follows that $\frac{1-\langle \nu,w\a_\text{high}^\vee\rangle}{1+\langle \nu,w\a_\text{high}^\vee\rangle}\rightarrow -1$ in this limit as well.  In such a case, the full $\frak p$-isotypic space of flat sections is in the $-1$ eigenspace of the limiting Hermitian form.  This phenomenon happens for 715 of the $960=120\times8$ pairs of characters $\d$ in the Weyl orbit of $\d_{\varpi_8}$ and perturbation directions $\b$.

\begin{proof}[Proof of \thmref{thm:easyalcove}]
Take $w$ to be the identity element, so that the restriction of the Hermitian form $\frac{1-\langle \nu,\a_\text{high}^\vee\rangle}{1+\langle \nu,\a_\text{high}^\vee\rangle}$ to  the $\frak p$-isotypic subspace is negative on dominant $\nu$ for which condition (\ref{fundalcove}) fails.  This proves the non-unitarity assertion.

On the other hand, unitarily-induced principal series are always unitary for $\nu=0$.  Since principal series of $E_{8(8)}$ have a unique lowest $K$-type, a deformation of the principal series along the family $\nu=t\rho$ for small $t\ge 0$ does not encounter any reducibility points.  It follows that unitarity is preserved for such a small deformation, which in particular immediately enters the open fundamental alcove that is defined through the inequality (\ref{fundalcove}).  Unitarity on the full open  alcove follows, since the signature of the restriction of the Hermitian form to $K$-types is constant on alcove facets.  Finally, unitarity extends to the closure of any unitary alcove, proving the Theorem.
\end{proof}

{\bf Applications of the vector-spinor.}  The 1920-dimensional $K$-type is even more powerful for proving non-unitarity.  It occurs with multiplicity 7 for each of the 120 characters $\delta$ in the orbit of $\d_{\varpi_8}$, and with multiplicity 8 for the other 135 characters $\d$.  In all but one case (the limit for $\beta=\alpha_2$ and $\delta=\d_{\varpi_8}$), the limiting Hermitian form has a nontrivial $-1$ eigenspace.  This exceptional $\d$ is precisely the one covered in~\thmref{thm:easyalcove}.  The proof of Theorem~\ref{mainthm} is therefore complete, since the FPP contains the fundamental alcove (because the highest root has a nontrivial coefficient of any simple coroot).

\section{The spherical case}\label{sec:spherical}

In this section, we present an application of limiting  intertwining integrals to give a new proof of Barbasch's theorem that the FPP inequality (\ref{barbasch1})-(\ref{barbasch2}) holds for spherical representations, i.e., dominant quotients of $V_{\nu,\delta_\text{triv}}$.  The proof is general but for brevity we will write it in the case of $E_{8(8)}$, where it simplifies somewhat because of the absence of diagram automorphisms.  However, we begin with a more general result about the $K$-type $\frak p$ under the adjoint action of $K$.

\begin{lem}\label{lem:petite}
Let $G$ be a split form of type $A$, $D$, or $E$.  Then $\frak p$ is petite.
\end{lem}

\noindent The assumption ensures that all roots have equal length, and lie in a common Weyl group orbit.  In fact $\frak p$ can fail to be petite for groups such as $G_{2(2)}$ which have multiple root lengths (though this  failure is not seen when studying spherical representations alone).

\begin{proof}
 Since $Z_\a$ and $iH_\a$ are related by a Cayley transform inside the complex root ${\frak sl}(2)$ for $\alpha$, the fact that $\frak p$ is petite  follows from the fact that $H_\a$'s spectrum under the adjoint action is bounded by 2.  Indeed, since the Weyl group acts transitively on both roots and coroots by an orthogonal transformation,  the bound follows from the Cauchy-Schwarz inequality.
\end{proof}

We now return to the setting of $G=E_{8(8)}$, and
identify $\frak p$ with its dual via the Killing form.
The discussion of matrix coefficients following (\ref{matrixcoefficients}) shows that
the flat sections in  (\ref{twoinds})  are spanned by  $F_{u',u}$ for which $M$ acts trivially on $u'$.  In our setting, those $u'$ correspond to elements $H\in\frak p$ which are invariant under the adjoint action of $M$.  Recall the decomposition (\ref{frakpdef}).  The $M$-invariance holds for any $H\in \frak h$, but not for any other nonzero linear combinations of the spanning elements $X_\a+X_{-\a}$, $\a\in \D_+$, because each $m\in M$ acts with sign $m^\alpha$ on $X_{\pm \a}$ -- a sign which cannot be trivial for all $m\in M$.   Thus the space of $M$-invariants can be precisely identified with $\frak h^*_\C$, whose dimension is indeed the rank of $G$ as   was pointed out in the last section.\\

With this background in mind, we turn to the proof itself.
Given $\beta\in\Sigma$, take $f_\nu$ to be any spherical flat section whose restriction to $K$ is the matrix coefficient $F_{u'_\beta,u}$,
where $u'_\beta$ corresponds to $\varpi_\beta$ and $u$ is arbitrary.
We will now argue that the limit in (\ref{limittheorem}) is $-f_\nu$.  First, $R(w_{\text{long},\beta},\nu,\delta)f_\nu=f_\nu$ because $f_\nu$ is left-invariant under $K\cap L'_\beta$, where $L'_\beta$ is the semisimple part   of $L_\beta$:~by
construction the normalized intertwining operator $R(w_{\text{long},\beta},\nu,\delta)$ defined in (\ref{normalizedintertwiners}) acts as the identity here, since its computation takes place on $L'_\beta$, on which $f_\nu$ is essentially spherical (see property 4 following (\ref{multrule})).

Recall that the constant $\varepsilon$ in (\ref{limittheorem}) is 1 when $\delta$ is trivial.  To finish the proof, we will show that $\Lambda(w_{\text{long}}(w_{\text{long},\beta})\i)$ sends $f_\nu$ to $-f_\nu$.  Note that $w_{\text{long},\beta}$ fixes $\varpi_\beta$, viewed as an element of $\frak h\cong \frak h^*$,  and $w_{\text{long}}$ acts as negation, since there are no nontrivial diagram automorphisms.
Thus under the adjoint action,
\begin{equation}\label{varpibetanegation}
 \tau'(w_\text{long} (w_{\text{long},\beta} )\i) u'_\beta \ \ = \ \ \tau'(w_{\text{long}}) u'_\beta  \ \ = \ \ -u'_\beta\,.
\end{equation}
Hence  $\Lambda(w_{\text{long}}(w_{\text{long},\beta})\i)f_\nu=-f_\nu$ and
it follows   that $f_\nu$ provides the requisite limiting $-1$ eigenvalue to prove the FPP inequality.  \bx

\end{document}